\documentclass[a4paper,12pt,reqno]{amsart}
\usepackage{eurosym}
\usepackage{amsfonts}
\usepackage{amsmath,amsthm,amssymb}
\usepackage{cite}

\nonstopmode \numberwithin{equation}{section}

\newtheorem{theorem}{Theorem}
\newtheorem{corollary}{Corollary}[section]

\allowdisplaybreaks

\begin{document}
\title[On fractional kinetic equations and their Sumudu transform...]{On fractional kinetic equations and their Sumudu transform multiparameter Struve functions based solutions}
\author{K.S. Nisar*, F.B.M. Belgacem, M. S. Abouzaid}
\address{K. S. Nisar : Department of Mathematics, College of Arts and
Science, Prince Sattam bin Abdulaziz University, Wadi Aldawaser, Riyadh
region 11991, Saudi Arabia}
\email{ksnisar1@gmail.com }
\address{F.B.M. Belgacem : Department of Mathematics,  Faculty of Basic Education, PAAET, Al-Ardhiya, Kuwait}
\email{fbmbelgacem@gmail.com}

\address{M. S. Abouzaid : Department of Mathematics, Faculty of Science, Kafrelshiekh University, Egypt. }
\email{moheb\_abouzaid@hotmail.com}
\thanks{*corresponding author}

\begin{abstract}
This research paper treats fractional kinetic equations using the Sumudu transform operator. The exact solutions obtained are presented in terms of Struve functions of four parameters.  By way of obtaining solutions some novel and useful and novel kinetic theorems are presented in light of the Sumudu properties.  Results obtained in this study may be pragmatically used in many branches theoretical and experimental science applications, not the least of which are mathematical physics, and various engineering fields.
\end{abstract}

\maketitle

\section{Introduction}

The Struve function $H_{v}\left( z\right) $  and modified Struve function $L_{v}\left( z\right) $ 
are are given by the following infinite series, respectively, 
\begin{equation}
H_{\upsilon }\left( z\right) =\left( \frac{z}{2}\right) ^{\upsilon
+1}\sum\limits_{k=0}^{\infty }\frac{\left( -1\right) ^{k}}{\Gamma \left( k+%
\frac{3}{2}\right) \Gamma \left( k+\upsilon +\frac{3}{2}\right) }\left( 
\frac{z}{2}\right) ^{2k},  \label{eqn-1-Struve}
\end{equation}%
and 
\begin{equation}
L_{\upsilon }\left( z\right) =\left( \frac{z}{2}\right) ^{\upsilon
+1}\sum\limits_{k=0}^{\infty }\frac{1}{\Gamma \left( k+\frac{3}{2}\right)
\Gamma \left( k+\upsilon +\frac{3}{2}\right) }\left( \frac{z}{2}\right) ^{2k}.
\label{eqn-2-Struve}
\end{equation}%
The Struve function of order $\upsilon $
(see \cite{Struve}) is the solution of the non-homogeneous Bessel differential equation: 
\begin{equation}
x^{2}y^{\prime \prime }(x)+xy^{^{\prime }}\left( x\right) +\left(
x^{2}-\upsilon^{2}\right) y\left( x\right) =\frac{4(x/2)^{\upsilon +1}}{\sqrt{\pi }%
\Gamma \left( {\upsilon +1/2}\right) },\quad \upsilon \in \mathbb{C}.
\label{StruveH}
\end{equation}%
where $\Gamma\left(z\right)$ is the gamma function and, $L_{\upsilon }\left( z\right)$ is
related to $H_{\upsilon }\left( z\right)$, by the relation
\begin{align}
L_{n}\left( z\right)=-ie^{-n\pi i/2}H_{n}\left(i z\right).
\end{align}
While, the homogeneous solutions of \eqref{StruveH} are the Bessel functions, the particular solutions may be given to  correspond to Struve functions.  For complex parameter, $\upsilon$,  modified Struve functions, $L_{\upsilon}$,  turn out to be solutions of a modified version of equation \eqref{StruveH},  where the LHS zeroth coefficient is replaced by: $-(x^2 + v^2)$.  Applications of Struve functions can be found in various branches of applied science (see \cite{Ahmedi, Hirata,Shaw,Shao, Pedersen}).

Generalized versions of Struve function are done by extending its domain or expanding the type and number of parameters.  
In particular, a four- parameter  generalized struve function, studied by Singh \cite{Singh1}, is defined by: 
\begin{equation}
H_{p,\mu }^{\lambda ,\alpha }\left( x\right) :=\sum_{k=0}^{\infty }\frac{%
\left( -1\right) ^{k}}{\Gamma \left( \alpha k+\mu \right) \Gamma \left(
\lambda k+p+\frac{3}{2}\right) }\left( \tfrac{x}{2}\right)
^{2k+p+1},p\in \mathbb{C}  \label{1}
\end{equation}%
where $\lambda >0,\alpha >0$ and $\mu $ an arbitrary parameter.

For subsequent need towards our objective below we also recall the generalized Mittag-Leffler function \cite{Mittag} defined by, 
\begin{equation}
E_{\alpha ,\beta }\left( x\right) =\sum_{n=0}^{\infty }\frac{x^{n}}{\Gamma
\left( \alpha n+\beta \right) }.  \label{Mittag-eqn}
\end{equation}%
The main motive of this paper is to study the solution of generalized
form of the fractional kinetic equation involving generalized Struve
function of four parameters with the help of Sumudu transform.

The Sumudu transform introduced by Watugala (see \cite{Watugala1, Watugala2}). 
For more details about Sumudu transform, see (\cite{Asiru,
Belgacem2003,Belgacem2005, Belgacem2006a,Belgacem2006b,Belgacem2010,Belgacem2009,Belgacem2017,Belgacem2016, CDB2012, HB}).
The Sumudu transform over the set functions
\begin{equation*}
A=\left\{ f\left( t\right) \left\vert \exists ~M,\tau _{1},\tau
_{2}>0,\left\vert f\left( t\right) \right\vert <Me^{\left\vert t\right\vert
/\tau _{j}}\right. ,~\text{if }t\in \left( -1\right) ^{j}\times \lbrack
0,\infty )\right\}
\end{equation*}

is defined by
\begin{equation}
G\left( u\right) =S\left[ f\left( t\right) ;u\right] =\int_{0}^{\infty
}f\left( ut\right) e^{-t}dt,~u\in \left( -\tau _{1},\tau _{2}\right) .
\label{S1}
\end{equation}

The significance of fractional differential equations in the field of
applied science increased more attention not only in mathematics but also in
physics, dynamical systems, control systems and systems engineering, to
create the mathematical model of numerous physical phenomena 
(\cite{Chaurasia, Chouhan,  Chouhan1, Gupta, Gupta1, Nisar1, Nisar2, Nisar3, 
Nisar4, Saichev, Saxena1, Saxena2, Saxena3, Saxena4, Zaslavsky}. 
To carry our investigation,we would like to recall the following results due to Haubold and Mathai \cite%
{Haubold}. The fractional differential equation between rate of change of
reaction was established by Haubold and Mathai \cite{Haubold}. The
destruction rate and the production rate are given as follows 
\begin{equation}
\frac{dN}{dt}=-d\left( N_{t}\right) +p\left( N_{t}\right)
\label{eqn-6-Struve}
\end{equation}
where $N=N\left( t\right) $ the rate of reaction ,$d=d\left( N\right) $ the
rate of destruction, $p=p\left( N\right) $ the rate of production and $N_{t}$
denote the function defined by $N_{t}\left( t^{\ast }\right)
=N\left(t-t^{\ast }\right) ,t^{\ast }>0$

The special case of ($\ref{eqn-6-Struve}$), for spatial fluctuations or in
homogeneities in $N\left( t\right) $ \ the quantity are neglected, that is
the equation 
\begin{equation}
\frac{dN}{dt}=-c_{i}N_{i}\left( t\right)  \label{eqn-7-Struve}
\end{equation}
with the initial condition that $N_{i}\left( t=0\right) =N_{0}$ is the
number of density of species $i$ at time $t=0$ and $c_{i}>0.$ If we reject
the index $i$ and integrate the standard kinetic equation ($\ref%
{eqn-7-Struve}$), we have 
\begin{equation}
N\left( t\right) -N_{0}=-c_{0}D_{t}^{-1}N\left( t\right)
\label{eqn-8-Struve}
\end{equation}
where $_{0}D_{t}^{-1}$ is the special case of the Riemann-Liouville integral
operator $_{0}D_{t}^{-\upsilon }$ defined as 
\begin{equation}
_{0}D_{t}^{-\upsilon }f\left( t\right) =\frac{1}{\Gamma \left( \upsilon
\right) }\int\limits_{0}^{t}\left( t-s\right) ^{\upsilon -1}f\left( s\right)
ds,t>0,\Re\left( \upsilon \right) >0  \label{eqn-8b-Struve}
\end{equation}%
The fractional generalization of the standard kinetic equation ($\ref%
{eqn-8-Struve}$) given by Haubold and Mathai \cite{Haubold} as: 
\begin{equation}
N\left( t\right) -N_{0}=-c_{0}^{\upsilon }D_{t}^{-1}N\left( t\right) 
\label{eqn-9-Struve}
\end{equation}
and obtained the solution of ($\ref{eqn-9-Struve}$) as follows 
\begin{equation}
N\left( t\right) =N_{0}\sum\limits_{k=0}^{\infty }\frac{\left( -1\right) ^{k}%
}{\Gamma \left( \upsilon k+1\right) }\left( ct\right) ^{\upsilon k}
\label{eqn-10-Struve}
\end{equation}
Further, Saxena and Kalla \cite{Saxena4} considered the following fractional
kinetic equation: 
\begin{equation}
N\left( t\right) -N_{0}f\left( t\right) =-c^{\upsilon }\left(
_{0}D_{t}^{-1}N\right) \left( t\right) \text{ \ \ \ }\left( \Re%
\left( \upsilon \right) >0\right)  \label{eqn-11-Struve}
\end{equation}
where $N(t)$ denotes the number density of a given species at time $t$, $%
N_{0}=N\left( 0\right) $ is the number density of that species at time $t=0$%
, $c$ is a constant and $f\in L(0,\infty $).
By applying the Laplace transform to ($\ref{eqn-11-Struve}$), 
\begin{align}
L\left[ N\left( t\right) \right] \left( p\right)&=N_{0}\frac{F\left(
p\right)}{1+c^{\upsilon }p^{-\upsilon }}\notag\\
&=N_{0}\left( \sum_{n=0}^{\infty}\left( -c^{\upsilon }\right) ^{n}p^{-n\upsilon }\right) F\left( p\right) 
\text{ \ \ }\left( n\in N_{0},\left\vert \frac{c}{p}\right\vert <1\right),
\label{eqn-12-Struve}
\end{align}
where the Laplace transform (\cite{Spiegel}) is defined by 
\begin{equation}
F\left( p\right) =L\left[ f(t)\right] =\int\limits_{0}^{\infty
}e^{-pt}f\left( t\right) dt\text{ \ \ }\Re\left( p\right) >0
\label{eqn-13-Struve}
\end{equation}
The Sumudu transform of  $\left( \ref{eqn-8b-Struve}\right) $ is (see, (\cite%
{Belgacem2003}), p.106,eqn (2.1)) :
\begin{equation}
S\left[ _{0}D_{t}^{-v}f\left( t\right) \right] =u^{v}G\left( u\right) 
\label{S2}
\end{equation}%
where $G\left( u\right) $ is defined in $\left( \ref{S1}\right) .$
It is clear that the Sumudu transform of $f\left( t\right) =t^{\alpha }$ is
given by 
\begin{equation}
S\left[ f\left( t\right) \right] =\int_{0}^{\infty }\left( tu\right)
^{\alpha }e^{-t}dt=u^{\alpha }\Gamma \left( \alpha +1\right) ,\Re%
\left( \alpha \right) >-1  \label{S3}
\end{equation}

\section{Solution of generalized fractional Kinetic equations}

We devote this section to the derivative of the exact solution of time
fractional kinetic equation via the well-known Sumudu transform. 
The results are as follows:
\begin{theorem}
\label{Th1} If $d>0,\upsilon >0,\alpha ,c,l,t\in \mathbb{C}$ ,$\Re%
\left( l\right) >-1$ and $\Re\left( u\right) >0$ with $\left\vert
u\right\vert <d^{-1},$ then the solution of the equation%
\begin{equation}
N\left( t\right) -N_{0}H_{l,\mu }^{\lambda ,\alpha }\left( t\right)
:=-d^{\upsilon }\text{ }_{0}D_{t}^{-\upsilon }N\left( t\right) ,
\label{eqn-14-Struve}
\end{equation}
\end{theorem}
is given by the following formula%
\begin{equation}
N\left( t\right) =\frac{N_{0}}{2}\sum_{k-0}^{\infty }\frac{\left( -1\right)
^{k}\Gamma \left( 2k+l+2\right) }{\Gamma \left( \alpha k+\mu \right) \Gamma
\left( \lambda k+l+\frac{3}{2}\right) }\left( \frac{t}{2}\right)
^{2k+l}E_{\upsilon ,2k+l+1}\left( -d^{\upsilon }t^{\upsilon }\right) 
\label{eqn-15-Struve}
\end{equation}%
where $E_{v,2k+l+1}\left( .\right) $ is the generalized Mittag-Leffler
function given in ($\ref{Mittag-eqn}$).

\begin{proof}
\label{Pf1} Taking the Sumudu transform to the both sides of ($\ref%
{eqn-14-Struve}$),  we obtain the following relation, 
\begin{equation*}
S\left\{ N\left( t\right) \right\} -N_{0}S\left\{ H_{l,\mu }^{\lambda
,\alpha }\left( t\right) ;p\right\} =-d^{\upsilon }S\left\{
_{0}D_{t}^{-\upsilon }N\left( t\right) \right\} 
\end{equation*}%
Now, using the definition given in  ($\ref{1}$), we get  
\begin{equation*}
\overset{\ast }{N}\left( u\right) -N_{0}\left\{ S\left( \sum_{k-0}^{\infty }%
\frac{\left( -1\right) ^{k}}{\Gamma \left( \alpha k+\mu \right) \Gamma
\left( \lambda k+l+\frac{3}{2}\right) }\left( \frac{t}{2}\right)
^{2k+l+1}\right) \right\} =-d^{\upsilon }u^{v}\overset{\ast }{N}\left(
u\right) 
\end{equation*}%
where $\overset{\ast }{N}\left( u\right) =S\left\{ N\left( t\right)
,u\right\} $ and $S\left\{ t^{\lambda -1}\right\} =u^{\lambda -1}\Gamma
\left( \lambda \right) $ gives,%
\begin{equation*}
\overset{\ast }{N}\left( u\right) -N_{0}\sum_{k-0}^{\infty }\frac{\left(
-1\right) ^{k}}{\Gamma \left( \alpha k+\mu \right) \Gamma \left( \lambda k+l+%
\frac{3}{2}\right) }u^{2k+l+1}\Gamma \left( 2k+l+2\right) =-d^{\upsilon
}u^{v}\overset{\ast }{N}\left( u\right) ,
\end{equation*}%
\begin{align*}
\overset{\ast }{N}\left( u\right) \left( 1+d^{\upsilon }u^{v}\right)&=N_{0}\sum_{k-0}^{\infty }\frac{\left( -1\right) ^{k}\left( \frac{u}{2}
\right) ^{2k+l+1}\Gamma \left( 2k+l+2\right) }{\Gamma \left( \alpha k+\mu
\right) \Gamma \left( \lambda k+l+\frac{3}{2}\right) } \\
\overset{\ast }{N}\left( u\right) & =\frac{N_{0}}{\left( 1+d^{\upsilon
}u^{v}\right) }\sum_{k-0}^{\infty }\frac{\left( -1\right) ^{k}\left( \frac{u
}{2}\right) ^{2k+l+1}\Gamma \left( 2k+l+2\right) }{\Gamma \left( \alpha
k+\mu \right) \Gamma \left( \lambda k+l+\frac{3}{2}\right) } \\
&=N_{0}\sum_{k-0}^{\infty }\frac{\left( -1\right) ^{k}\left( \frac{u}{2}
\right) ^{2k+l+1}\Gamma \left( 2k+l+2\right) }{\Gamma \left( \alpha k+\mu
\right) \Gamma \left( \lambda k+l+\frac{3}{2}\right) }\sum_{r-0}^{\infty
}\left( -1\right) ^{r}\left( du\right) ^{vr}
\end{align*}%
Now, taking the inverse transform of the above expression and using the
relation
\begin{align*} 
S^{-1}\left\{ u^{v}\right\} =\frac{t^{\lambda -1}}{\Gamma \left(v\right) },\quad\Re\left( u\right) >0,\Re\left( v\right) >0,
\end{align*} 
we get   
\begin{eqnarray*}
N\left( t\right)  &=&N_{0}\sum_{k-0}^{\infty }\frac{\left( -1\right)
^{k}\left( 2\right) ^{-\left( 2k+l+1\right) }\Gamma \left( 2k+l+2\right) }{%
\Gamma \left( \alpha k+\mu \right) \Gamma \left( \lambda k+l+\frac{3}{2}%
\right) }\sum_{r-0}^{\infty }\frac{\left( -1\right) ^{r}d^{vr}t^{2k+vr+l}}{%
\Gamma \left( 2k+l+1+vr\right) } \\
&=&N_{0}\sum_{k-0}^{\infty }\frac{\left( -1\right) ^{k}\Gamma \left(
2k+l+2\right) t^{2k+l}}{\Gamma \left( \alpha k+\mu \right) \Gamma \left(
\lambda k+l+\frac{3}{2}\right) \left( 2\right) ^{2k+l+1}}\sum_{r-0}^{\infty }%
\frac{\left( -1\right) ^{r}d^{vr}t^{vr}}{\Gamma \left( 2k+l+1+vr\right) }
\end{eqnarray*}

In view of $\left( \ref{Mittag-eqn}\right) $, we obtain the desired result.
\end{proof}

\begin{corollary}
\label{Cor1} If we put $\alpha =\lambda =1$ and $\mu =3/2$ in $\left( \ref%
{eqn-14-Struve}\right) $ then we obtain the solution of fractional kinetic
equation involving Struve function $H_{v}\left( z\right) $  as :
\end{corollary}

For $d>0,\upsilon >0,c,l,t\in \mathbb{C}$,$\Re\left( l\right) >-1$ 
and $\Re\left( u\right) >0,$ then the solution of the equation%
\begin{equation*}
N\left( t\right) -N_{0}H_{l,3/2}^{1,1}\left( t\right) :=-d^{\upsilon }\text{ 
}_{0}D_{t}^{-\upsilon }N\left( t\right) ,
\end{equation*}
is given by the following formula%
\begin{equation*}
N\left( t\right) =\frac{N_{0}}{2}\sum_{k-0}^{\infty }\frac{\left( -1\right)
^{k}\Gamma \left( 2k+l+2\right) }{\Gamma \left( k+l+\frac{3}{2}\right)
\Gamma \left( k+\frac{3}{2}\right) }\left( \frac{t}{2}\right)
^{2k+l}E_{v,2k+l+1}\left( -d^{\upsilon }t^{\upsilon }\right) 
\end{equation*}

\begin{theorem}
\label{Th2} If $d>0,\upsilon >0,\alpha ,c,b,l,t\in \mathbb{C}$, $\Re%
\left( l\right) >-1$ and $\Re\left( u\right) >0$  then for the
solution of the equation%
\begin{equation}
N\left( t\right) -N_{0}H_{p,\mu }^{\lambda ,\alpha }\left( d^{\upsilon
}t^{\upsilon }\right) =-d^{\upsilon }{}_{0}D_{t}^{-\upsilon }N\left(
t\right)   \label{eqn-18-Struve}
\end{equation}
there holds the formula%
\begin{align}
N\left( t\right) &=N_{0}\left( \frac{d^{v}}{2}\right)
^{l+1}t^{lv+v-1}\notag\\
&\times\sum\limits_{k=0}^{\infty }\frac{\left( -1\right) ^{k}\Gamma
\left( 2kv+vl+\nu +1\right) }{\Gamma \left( \alpha k+\mu \right) \Gamma
\left( \lambda k+l+\frac{3}{2}\right) }\left( \frac{d^{\upsilon }t^{v}}{2}%
\right) ^{2k}E_{v,\left( 2k+l+1\right) \upsilon }\left( -d^{\upsilon
}t^{\upsilon }\right),   \label{eqn-19-Struve}
\end{align}
where $E_{v,2k\upsilon +lv+v+1}\left( .\right) $ is the generalized
Mittag-Leffler function $\left( \ref{Mittag-eqn}\right) .$
\end{theorem}

\begin{proof}
The proof can be proved in parallel with the proof of Theorem \ref{Th1}, so
the details of proofs are omitted.
\end{proof}

\begin{corollary}
\label{Cor2} If we set, $\alpha =\lambda =1$ and $\mu =3/2$ in \eqref{eqn-18-Struve} then we obtain the solution of fractional kinetic equation
involving Struve function $H_{v}\left( z\right) $ as :
For $d>0,\upsilon >0,c,l,t\in \mathbb{C}$,  $\Re\left( l\right) >-1$
and $\Re\left( u\right) >0,$ then the solution of the equation $,$
\begin{equation*}
N\left( t\right) -N_{0}H_{l,3/2}^{1,1}\left( d^{v}t^{v}\right)
:=-d^{\upsilon }\text{ }_{0}D_{t}^{-\upsilon }N\left( t\right) ,
\end{equation*}
is given by 
\begin{align*}
N\left( t\right)&=N_{0}\left( \frac{d^{v}}{2}\right)
^{l+1}t^{lv+v-1}\notag\\
&\times\sum\limits_{k=0}^{\infty }\frac{\left( -1\right) ^{k}\Gamma
\left( 2kv+vl+\nu +1\right) }{\Gamma \left( k+\frac{3}{2}\right) \Gamma
\left( k+l+\frac{3}{2}\right) }\left( \frac{d^{\upsilon }t^{v}}{2}\right)
^{2k}E_{v,\left( 2k+l+1\right) \upsilon }\left( -d^{\upsilon }t^{\upsilon
}\right) 
\end{align*}
\end{corollary}

\begin{theorem}
\label{Th3} If $d>0,\upsilon >0,a,c,b,l,t\in \mathbb{C}$ , $\Re%
\left( l\right) >-1,\Re\left( u\right) $ $>0$ with $\left\vert
u\right\vert <d^{-1}$ and $\mathfrak{a}\neq d$ then for the solution of the
equation%
\begin{equation}
N\left( t\right) -N_{0}H_{l,\mu }^{\lambda ,\alpha }\left( d^{\upsilon
}t^{\upsilon }\right) =-\mathfrak{a}^{\upsilon }{}_{0}D_{t}^{-\upsilon
}N\left( t\right),   \label{eqn-23-Struve}
\end{equation}
there hold the formula%
\begin{eqnarray}
N\left( t\right)  &=&N_{0}\left( \frac{d^{v}}{2}\right)
^{l+1}t^{lv+v-1}\sum\limits_{k=0}^{\infty }\frac{\left( -1\right) ^{k}\Gamma
\left( 2k\upsilon +\upsilon l+\upsilon +1\right) }{\Gamma \left( \alpha
k+\mu \right) \Gamma \left( \lambda k+l+\frac{3}{2}\right) }  \notag \\
&&\times \left( \frac{d^{\upsilon }t^{v}}{2}\right) ^{2k}E_{v,\left(
2k+l+1\right) \upsilon }\left( -\mathfrak{a}^{\upsilon }t^{\upsilon }\right), 
\label{eqn-24-Struve}
\end{eqnarray}
where $E_{v,\left( 2k+l+1\right) \upsilon }\left( -\mathfrak{a}^{\upsilon
}t^{\upsilon }\right) $ is the Mittag-Leffler function in $\left( \ref%
{Mittag-eqn}\right) .$
\end{theorem}

\begin{proof}
The proof of theorem \ref{Th3} is derived similarly as that of theorems \ref%
{Th1} and \ref{Th2}.
\end{proof}

\begin{corollary}
\label{Cor3} By setting $\alpha =\lambda =1$ and $\mu =3/2$ in $\left( \ref%
{eqn-23-Struve}\right) $ we obtain the  solution of fractional kinetic
equations as :
For $d>0,\upsilon >0,c,l,t\in \mathbb{C}\mathbf{,}\mathfrak{a}\neq d$ , $%
\Re\left( l\right) >-1~\Re\left( u\right) >0~$and  then the
solution of the equation $,$
\begin{equation*}
N\left( t\right) -N_{0}H_{l,3/2}^{1,1}\left( d^{v}t^{v}\right) :=-\mathfrak{a%
}^{\upsilon }\text{ }_{0}D_{t}^{-\upsilon }N\left( t\right) ,
\end{equation*}
is given by 
\begin{eqnarray*}
N\left( t\right)  &=&N_{0}\left( \frac{d^{v}}{2}\right)
^{l+1}t^{lv+v-1}\sum\limits_{k=0}^{\infty }\frac{\left( -1\right) ^{k}\Gamma
\left( 2k\upsilon +\upsilon l+\upsilon +1\right) }{\Gamma \left( k+\frac{3}{2%
}\right) \Gamma \left( k+l+\frac{3}{2}\right) } \\
&&\times \left( \frac{d^{\upsilon }t^{v}}{2}\right) ^{2k}E_{v,\left(
2k+l+1\right) \upsilon }\left( -\mathfrak{a}^{\upsilon }t^{\upsilon }\right) 
\end{eqnarray*}
\end{corollary}

\subsection{Special Cases}

\begin{enumerate}
\item[1]. Consider the generalized Struve function given by Bhowmic \cite%
{Bhowmick} 
\begin{equation}
H_{l}^{\lambda }\left( x\right) =\sum_{k-0}^{\infty }\frac{\left( -1\right)
^{k}\left( \frac{t}{2}\right) ^{2k+l+1}}{\Gamma \left( \lambda k+l+\frac{3}{2%
}\right) \Gamma \left( k+\frac{3}{2}\right) }  \label{SC-1}
\end{equation}
\end{enumerate}
Now, we have the following corollaries due to theorems 1, 2 and
3 respectively.

\begin{corollary}
\label{Cor4} If $d>0,\upsilon >0,\lambda ,l,t\in \mathbb{C}$, $\Re%
\left( l\right) >-1$ and $\Re\left( u\right) >0$ then the solution
of the equation 
\begin{equation*}
N\left( t\right) -N_{0}H_{l}^{\lambda }\left( t\right) :=-d^{\upsilon }\text{
}_{0}D_{t}^{-\upsilon }N\left( t\right) ,
\end{equation*}
is given by the following formula%
\begin{equation*}
N\left( t\right) =\frac{N_{0}}{2}\sum_{k-0}^{\infty }\frac{\left( -1\right)
^{k}\Gamma \left( 2k+l+2\right) }{\Gamma \left( \lambda k+l+\frac{3}{2}%
\right) \Gamma \left( k+\frac{3}{2}\right) }\left( \frac{t}{2}\right)
^{2k+l}E_{\upsilon ,2k+l+1}\left( -d^{\upsilon }t^{\upsilon }\right) 
\end{equation*}
\end{corollary}

\begin{corollary}
\label{Cor5} If $d>0,\upsilon >0,\lambda ,l,t\in \mathbb{C}$ , $\Re%
\left( l\right) >-1$ and $\Re\left( u\right) >0$ then for the
solution of the equation%
\begin{equation*}
N\left( t\right) -N_{0}H_{l}^{\lambda }\left( d^{\upsilon }t^{\upsilon
}\right) =-d^{\upsilon }{}_{0}D_{t}^{-\upsilon }N\left( t\right) 
\end{equation*}
there holds the formula%
\begin{equation*}
N\left( t\right) =N_{0}\left( \frac{d^{v}}{2}\right)
^{l+1}t^{lv+v-1}\sum\limits_{k=0}^{\infty }\frac{\left( -1\right) ^{k}\Gamma
\left( 2kv+vl+\nu +1\right) }{\Gamma \left( \lambda k+l+\frac{3}{2}\right)
\Gamma \left( k+\frac{3}{2}\right) }\left( \frac{d^{\upsilon }t^{v}}{2}%
\right) ^{2k}E_{v,\left( 2k+l+1\right) \upsilon }\left( -d^{\upsilon
}t^{\upsilon }\right) 
\end{equation*}
\end{corollary}

\begin{corollary}
\label{Cor6} If $d>0,\upsilon >0,\lambda ,b,l,t\in \mathbb{C}$ ,$\mathfrak{a}%
\neq d$ ,$\Re\left( l\right) >-1$ and $\Re\left( u\right) >0$
then for the solution of the equation%
\begin{equation*}
N\left( t\right) -N_{0}H_{l}^{\lambda }\left( d^{\upsilon }t^{\upsilon
}\right) =-\mathfrak{a}^{\upsilon }{}_{0}D_{t}^{-\upsilon }N\left( t\right) 
\end{equation*}
there hold the formula%
\begin{eqnarray*}
N\left( t\right)  &=&N_{0}\left( \frac{d^{v}}{2}\right)
^{l+1}t^{lv+v-1}\sum\limits_{k=0}^{\infty }\frac{\left( -1\right) ^{k}\Gamma
\left( 2k\upsilon +\upsilon l+\upsilon +1\right) }{\left( \lambda k+l+\frac{3%
}{2}\right) \Gamma \left( k+\frac{3}{2}\right) } \\
&&\times \left( \frac{d^{\upsilon }t^{v}}{2}\right) ^{2k}E_{v,\left(
2k+l+1\right) \upsilon }\left( -\mathfrak{a}^{\upsilon }t^{\upsilon }\right) 
\end{eqnarray*}
\end{corollary}

\begin{enumerate}
\item[2]. Consider the generalized Struve function given by Kant \cite{Kant} 
\begin{equation}
H_{l}^{\lambda ,\alpha }\left( x\right) =\sum_{k-0}^{\infty }\frac{\left(
-1\right) ^{k}\left( \frac{x}{2}\right) ^{2k+l+1}}{\Gamma \left( \lambda k+l+%
\frac{3}{2}\right) \Gamma \left( \alpha k+\frac{3}{2}\right) }  \label{SC-2}
\end{equation}
\end{enumerate}

Now, we have the following corollaries due to theorem 1, 2 and 3 respectively.
\begin{corollary}
\label{Cor7} If $d>0,\upsilon >0,\lambda ,l,t\in \mathbb{C}$ , $\Re%
\left( l\right) >-1$ and $\Re\left( u\right) >0,$ then the solution
of the equation 
\begin{equation*}
N\left( t\right) -N_{0}H_{l}^{\lambda ,\alpha }\left( t\right)
:=-d^{\upsilon }\text{ }_{0}D_{t}^{-\upsilon }N\left( t\right) ,
\end{equation*}
is given by the following formula%
\begin{equation*}
N\left( t\right) =\frac{N_{0}}{2}\sum_{k-0}^{\infty }\frac{\left( -1\right)
^{k}\Gamma \left( 2k+l+2\right) }{\left( \lambda k+l+\frac{3}{2}\right)
\Gamma \left( \alpha k+\frac{3}{2}\right) }\left( \frac{t}{2}\right)
^{2k+l}E_{\upsilon ,2k+l+1}\left( -d^{\upsilon }t^{\upsilon }\right) 
\end{equation*}
\end{corollary}

\begin{corollary}
\label{Cor8} If $d>0,\upsilon >0,\lambda ,l,t\in \mathbb{C}$ ,$\Re%
\left( l\right) >-1$ and $\Re\left( u\right) >0$ then for the
solution of the equation%
\begin{equation*}
N\left( t\right) -N_{0}H_{l}^{\lambda ,\alpha }\left( d^{\upsilon
}t^{\upsilon }\right) =-d^{\upsilon }{}_{0}D_{t}^{-\upsilon }N\left(
t\right) 
\end{equation*}
there holds the formula%
\begin{equation*}
N\left( t\right) =N_{0}\left( \frac{d^{v}}{2}\right)
^{l+1}t^{lv+v-1}\sum\limits_{k=0}^{\infty }\frac{\left( -1\right) ^{k}\Gamma
\left( 2kv+vl+\nu +1\right) }{\Gamma \left( \lambda k+l+\frac{3}{2}\right)
\Gamma \left( \alpha k+\frac{3}{2}\right) }\left( \frac{d^{\upsilon }t^{v}}{2%
}\right) ^{2k}E_{v,\left( 2k+l+1\right) \upsilon }\left( -d^{\upsilon
}t^{\upsilon }\right) 
\end{equation*}
\end{corollary}

\begin{corollary}
\label{Cor9}~If $d>0,\upsilon >0,\lambda ,b,l,t\in \mathbb{C}$ ,$\mathfrak{a}%
\neq d$ ,$\Re\left( l\right) >-1$ and $\Re\left( u\right) >0$
then for the solution of the equation%
\begin{equation*}
N\left( t\right) -N_{0}H_{l}^{\lambda ,\alpha }\left( d^{\upsilon
}t^{\upsilon }\right) =-\mathfrak{a}^{\upsilon }{}_{0}D_{t}^{-\upsilon
}N\left( t\right) 
\end{equation*}
there hold the formula%
\begin{eqnarray*}
N\left( t\right)  &=&N_{0}\left( \frac{d^{v}}{2}\right)
^{l+1}t^{lv+v-1}\sum\limits_{k=0}^{\infty }\frac{\left( -1\right) ^{k}\Gamma
\left( 2k\upsilon +\upsilon l+\upsilon +1\right) }{\left( \lambda k+l+\frac{3%
}{2}\right) \Gamma \left( \alpha k+\frac{3}{2}\right) } \\
&&\times \left( \frac{d^{\upsilon }t^{v}}{2}\right) ^{2k}E_{v,\left(
2k+l+1\right) \upsilon }\left( -\mathfrak{a}^{\upsilon }t^{\upsilon }\right) 
\end{eqnarray*}
\end{corollary}

\begin{enumerate}
\item[3]. Consider the generalized Struve function given by Singh \cite%
{Singh1} 
\begin{equation}
H_{l,\mu }^{\lambda }\left( x\right) =\sum_{k-0}^{\infty }\frac{\left(
-1\right) ^{k}\left( \frac{x}{2}\right) ^{2k+l+1}}{\Gamma \left( \lambda k+%
\frac{l}{\mu }+\frac{3}{2}\right) \Gamma \left( k+\frac{3}{2}\right) }
\label{SC-3}
\end{equation}
\end{enumerate}

Now, we have the following corollaries due to theorem 1, theorem 2 and
theorem 3 respectively.
\begin{corollary}
\label{Cor10} If $d>0,\upsilon >0,\lambda ,l,t\in \mathbb{C}$ $,\Re%
\left( l\right) >-1$ and $\Re\left( u\right) >0$ then the solution
of the equation 
\begin{equation*}
N\left( t\right) -N_{0}H_{l,\mu }^{\lambda }\left( t\right) :=-d^{\upsilon }%
\text{ }_{0}D_{t}^{-\upsilon }N\left( t\right) ,
\end{equation*}
is given by the following formula%
\begin{equation*}
N\left( t\right) =\frac{N_{0}}{2}\sum_{k-0}^{\infty }\frac{\left( -1\right)
^{k}\Gamma \left( 2k+l+2\right) }{\Gamma \left( \lambda k+\frac{l}{\mu }+%
\frac{3}{2}\right) \Gamma \left( k+\frac{3}{2}\right) }\left( \frac{t}{2}%
\right) ^{2k+l}E_{\upsilon ,2k+l+1}\left( -d^{\upsilon }t^{\upsilon }\right) 
\end{equation*}
\end{corollary}

\begin{corollary}
\label{Cor11} If $d>0,\upsilon >0,\lambda ,l,t\in \mathbb{C}$ $,\Re%
\left( l\right) >-1$ and $\Re\left( u\right) >0$ then for the
solution of the equation%
\begin{equation*}
N\left( t\right) -N_{0}H_{l,\mu }^{\lambda }\left( d^{\upsilon }t^{\upsilon
}\right) =-d^{\upsilon }{}_{0}D_{t}^{-\upsilon }N\left( t\right) 
\end{equation*}
there holds the formula%
\begin{equation*}
N\left( t\right) =N_{0}\left( \frac{d^{v}}{2}\right)
^{l+1}t^{lv+v-1}\sum\limits_{k=0}^{\infty }\frac{\left( -1\right) ^{k}\Gamma
\left( 2kv+vl+\nu +1\right) }{\Gamma \left( \lambda k+\frac{l}{\mu }+\frac{3%
}{2}\right) \Gamma \left( k+\frac{3}{2}\right) }\left( \frac{d^{\upsilon
}t^{v}}{2}\right) ^{2k}E_{v,\left( 2k+l+1\right) \upsilon }\left(
-d^{\upsilon }t^{\upsilon }\right). 
\end{equation*}
\end{corollary}

\begin{corollary}
\label{Cor12} If $d>0,\upsilon >0,\lambda ,b,l,t\in \mathbb{C}$ ,$\mathfrak{a%
}\neq d$ $,\Re\left( l\right) >-1,\mu $ is an arbitrary parameter
and $\Re\left( u\right) >0$ then for the solution of the equation%
\begin{equation*}
N\left( t\right) -N_{0}H_{l,\mu }^{\lambda }\left( d^{\upsilon }t^{\upsilon
}\right) =-\mathfrak{a}^{\upsilon }{}_{0}D_{t}^{-\upsilon }N\left( t\right) 
\end{equation*}
there hold the formula%
\begin{eqnarray*}
N\left( t\right)  &=&N_{0}\left( \frac{d^{v}}{2}\right)
^{l+1}t^{lv+v-1}\sum\limits_{k=0}^{\infty }\frac{\left( -1\right) ^{k}\Gamma
\left( 2k\upsilon +\upsilon l+\upsilon +1\right) }{\left( \lambda k+\frac{l}{%
\mu }+\frac{3}{2}\right) \Gamma \left( k+\frac{3}{2}\right) } \\
&&\times \left( \frac{d^{\upsilon }t^{v}}{2}\right) ^{2k}E_{v,\left(
2k+l+1\right) \upsilon }\left( -\mathfrak{a}^{\upsilon }t^{\upsilon }\right) 
\end{eqnarray*}
\end{corollary}

\end{document}